\documentclass[12pt]{article}
\usepackage{amsmath}
\usepackage{amsfonts,amssymb}
\usepackage{extarrows}
\usepackage{geometry}
\usepackage{authblk} 
\usepackage{hyperref} 
\usepackage{enumitem}
\usepackage[utf8]{inputenc}
\newtheorem{theorem}{Theorem}[section]

\newtheorem{lemma}[theorem]{Lemma}
\newtheorem{claim}[theorem]{Claim}

\newenvironment{proof}{{\noindent\it Proof.}\quad}{\hfill $\square$\par}

\usepackage{graphicx}

\usepackage{float}

\newtheorem{problem}[theorem]{Problem}

\usepackage{
	amsmath,			
	amssymb,			
	graphicx,			
	lastpage,			
	multicol,			
	multirow,			
	pifont,			    
	enumitem
}

\usepackage[numbers]{natbib}

\begin{document}


\title{Spectral extrema of graphs of given even size forbidding $H(4,3)$}
\author[1]{\small\bf Ruiling Zheng\thanks{E-mail:  rlzheng2017@163.com}}
\author[2,3]{\small\bf Gang Zhang\thanks{E-mail:  gzh\_ang@163.com}}


\affil[1]{\small Center for Combinatorics and LPMC, Nankai University, China}
\affil[2]{\small School of Mathematical Sciences, Xiamen University, China}
\affil[3]{\small Institute of Mathematics, Physics and Mechanics, Ljubljana, Slovenia}
\date{}

\maketitle\baselineskip 16.3pt

\begin{abstract}
A graph is sad to be $H$-free if it does not contain $H$ as a subgraph. Let $H(k,3)$ be the graph formed by taking a cycle of length $k$ and a triangle on a common vertex. Li, Lu and Peng [Discrete Math. 346 (2023) 113680] proved that if $G$ is an $H(3,3)$-free graph of size $m \geq 8$, then the spectral radius $\rho(G) \leq \frac{1+\sqrt{4 m-3}}{2}$ with equality if and only if $G \cong S_{\frac{m+3}{2}, 2}$, where $S_{\frac{m+3}{2}, 2}=K_2 \vee \frac{m-1}{2}K_1$. Note that the bound is attainable only when $m$ is odd. 
Recently, Pirzada and Rehman [Comput. Appl. Math. 44 (2025) 295] proved that if $G$ is an $\{H(3,3),H(4,3)\}$-free graph of even size $m \geq 10$, then $\rho(G) \leq \rho^{\prime}(m)$ with equality if and only if $G \cong S_{\frac{m+4}{2}, 2}^{-}$, where $\rho^{\prime}(m)$ is the largest root of $x^4-m x^2-(m-2) x+\frac{m}{2}-1=0$, and $S_{\frac{m+4}{2}, 2}^{-}$ is the graph obtained from $S_{\frac{m+4}{2}, 2}$ by deleting an edge incident to a vertex of degree two. In this paper, we improve the result of Pirzada and Rehman by showing that if $G$ is an $H(4,3)$-free graph of even size $m \geq 38$ without isolated vertices, then $\rho(G) \leq \rho^{\prime}(m)$ with equality if and only if $G \cong S_{\frac{m+4}{2}, 2}^{-}$.
\end{abstract}


\vskip.3cm

\section{Introduction}
Let $G=(V, E)$ be a simple undirected graph of {\it order} $n=|V(G)|$ and {\it size} $m=|E(G)|$. Let $A(G)$ be the adjacency matrix of $G$. Clearly, $A(G)$ is real symmetric. Hence, its eigenvalues are
real and we can arrange them as $\lambda_1(G) \geqslant \cdots \geqslant \lambda_n(G)$. The spectral radius, $\lambda(G)$, of $G$ is $\max \left\{\left|\lambda_1(G)\right|, \ldots,\left|\lambda_n(G)\right|\right\}$. Actually, it is equal to $\lambda_1(G)$ by Perron Frobenius Theorem. Moreover, there exists a unique positive eigenvector corresponding to $\rho(G)$, which is called the Perron vector of $G$.  
 For any two subsets $S,T \subseteq V(G)$, let $G[S]$ and $G-S$ denote the subgraphs of $G$ induced by $S$ and $V(G)\setminus S$, respectively. We denote by $e(S, T)$ the number of edges with one end in $S$ and the other end in $T$. For brevity, write $e(S)$ instead of $e(S, S)$. Let $N(v)$ be the set of neighbors of $v$, and $N[v]=N(v) \cup\{v\}$. Moreover, let $N_S(v)$ be the set of neighbors of $v$ in $S$ and denote $d_S(v)=\left|N_S(v)\right|$. 
 Given two vertex-disjoint graphs $G$ and $H$, $G \cup H$ denotes the {\it union} of $G$ and $H$, and $G \vee H$ denotes the {\it join} of $G$ and $H$. The graph $tG$ denotes the union of $t$ copies of $G$. Let $P_n$, $C_n$, $K_n$ and $K_{t,n-t}$ denote the path, cycle, complete graph and complete bipartite graph of order $n$, respectively. Let $S_{n, 2}$ be the graph of order $n$ and size $m$ obtained by joining each vertex of $K_2$ to $n-2$ isolated vertices. Note that $n=\frac{m+3}{2}$, and $S_{n, 2}=S_{\frac{m+3}{2}, 2}=K_2 \vee \frac{m-1}{2}K_1$. Let $S_{\frac{m+3}{2}, 2}^{-}$ be the graph obtained from $S_{\frac{m+3}{2}, 2}$ by deleting an edge incident to a vertex of degree 2.

Given a graph $H$, the graph $G$ is said to be {\it $H$-free} if $G$ does not contain $H$ as a subgraph. In 2010, Nikiforov \cite{Nikiforov2010} introduced the classic spectral Tur\'an type problem, also known as the Brualdi-Solheid-Tur\'an problem \cite{Liu2023,Nikiforov2007,Wilf1986,Zhai2022}. This problem is to determine the maximum spectral radius of an $H$-free graph with given order $n$. It is natural to ask what is the maximum spectral radius of an $H$-free graph with given size $m$, which is referred to as the Brualdi-Hoffman-Tur\'an problem \cite{Brualdi1985,LiZhaiShu2024,Lin2021,Nosal1970}.

 Let $F_{k,3}$ be the {\it friendship graph} obtained from $k$ triangles by sharing a common vertex. In 2023, Li, Lu and Peng \cite{LiLuPeng2023} proved that for any $F_{2,3}$-free graph $G$ of size $m \geq 8$, $\rho(G) \leq \frac{1+\sqrt{4 m-3}}{2}$ and the unique spectral extremal graph is $S_{\frac{m+3}{2}, 2}$. However, it is worth noting that this bound is not attainable when $m$ is even. Chen et al. \cite{Chen2025} proved that if $G$ is an $H(3,3)$-free graph of even size $m \geq 16$,  then $\rho(G) \leq$ $\rho^{\prime}(m)$, where $\rho^{\prime}(m)$ is the largest root of $x^4-m x^2-(m-2) x+\frac{m}{2}-1=0$. Min, Lou and Huang \cite{Min2022} pointed out $\rho^{\prime}(m)=\rho(S_{\frac{m+4}{2}, 2}^{-})$. In \cite{Chen2025}, the authors also proved that their equality holds if and only if $G \cong S_{\frac{m+4}{2}, 2}^{-}$.
 
 Let $H(k,l)$ be the graph formed by taking two cycles of lengths $k$ and $l$ on a common vertex. Note that $F_{2,3}=H(3,3)$. In \cite{Pirzada2024},  Rehman and Pirzada obtained that if $G$ is an $\{H(3,3), H(4,3)\}$-free graph of odd size, then $\rho(G) \leq \frac{1+\sqrt{4 m-3}}{2}$, with equality if and only if $G \cong S_{\frac{m+3}{2}, 2}$. Recently, they \cite{Pirzada2025} shown that if $G$ is an $\{H(3,3), H(4,3)\}$-free graph of even size $m \geq 10$, then $\rho(G) \leq$ $\rho^{\prime}(m)$.  Furthermore, the result of \cite{LiLuPeng2023} improved by Zhang and Wang \cite{Zhang2024}, and Yu, Li and Peng \cite{Yu2025} independently.

The aim of this paper is to determine the maximum spectral radius of graphs of given (even) size with a forbidden subgraph $H(4,3)$. Our main theorem is as follows, a proof of which is presented in Section \ref{sec:proof}. We remark that Theorem \ref{thm:even-case} generalizes the result of Rehman and Pirzada \cite{Pirzada2025}.
 
\begin{theorem}
    \label{thm:even-case}
    If $G$ is an $H(4,3)$-free graph of even size $m \geq 38$ without isolated vertices, then $\rho(G) \le \rho'(m)$ with equality if and only if $G \cong S_{\frac{m+4}{2}, 2}^{-}$.
\end{theorem}

Let $F_{t}=K_1\vee P_{t}$ be the $fan graph$. Very recently, the spectral extrema result for the graph $F_5$ was obtained by Gao and Li \cite{Gao2026}. Specifically, they proved that if $G$ is an $F_5$-free graph of size $m \geq 88$ without isolated vertices, then $\rho(G) \le \frac{1 + \sqrt{4m - 3}}{2}$ and the equality holds if and only if $G \cong S_{\frac{m+3}{2}, 2}$. Since $H(4,3)$ is a subgraph of $F_5$, the same result applies to $H(4,3)$-free graphs. Note that $S_{\frac{m+3}{2}, 2}$ contains no subgraph isomorphic to $H(4,3)$. Then the bound $\frac{1 + \sqrt{4m - 3}}{2}$ is attainable for $H(4,3)$-free graphs of odd size $m$. For more results on the Brualdi-Hoffman-Tur\'an problem for the fan graphs, see \cite{ChenYuan2025,LiZhaoZou2025,Liu2025,Zhang2025}.

The following is an immediate consequence of the Gao-Li's result and Theorem \ref{thm:even-case}.

\begin{theorem}
    \label{thm:main-theorem}
    If $G$ is an $H(4,3)$-free graph of size $m \geq 88$ without isolated vertices, then

    \[
    \rho(G) \le 
    \begin{cases}
        \frac{1 + \sqrt{4m - 3}}{2}, & \text{if $m$ is odd}\,, \\[6pt]
        \rho'(m), & \text{if $m$ is even}\,.
    \end{cases}
    \]

    \noindent Moreover, the equality is attained if and only if 
    \[
    G \cong 
    \begin{cases}
        S_{\frac{m+3}{2}, 2}, & \text{if $m$ is odd}\,, \\[6pt]
        S_{\frac{m+4}{2}, 2}^{-}, & \text{if $m$ is even}\,.
    \end{cases}
    \]
\end{theorem}

\section{Preliminaries}

In this section, we shall give some preliminary results used to prove our main theorem.

\begin{lemma}[Nikiforov \cite{Nikiforov2009}]
\label{lem:Nikiforov}
Let $A$ and $A^{\prime}$ be the adjacency matrices of two graphs $G$ and $G^{\prime}$ on the same vertex set. Suppose that $N_G(u) \subsetneq N_{G^{\prime}}(u)$ for some vertex $u$. If the Perron vector of $A(G)$ corresponding to $\rho(G)$ satisfies $X^T A^{\prime} X \geq X^T A X$, then $\rho\left(G^{\prime}\right)>\rho(G)$.
\end{lemma}

\begin{lemma}[\cite{Min2022}]
\label{lem:extremal-graphs-spectral-radius}
For $m \geq 6$, we have $\rho\left(S_{\frac{m+4}{2}, 2}^{-}\right)>\frac{1+\sqrt{4 m-5}}{2}$.
\end{lemma}

Denote $F^{\prime}_{m,l}$ the graph of size $m$ obtained by joining one of the vertices with maximum
degree in $S_{\frac{m-l+3}{2}, 2}$ to $l$ isolated vertices. Here we should point out that in \cite{Fang2023} $F_{m,l}$ is usually used for the above graphs. However, we also have the friendship graph $F_{k,3}$. In order to avoid confusion, we adopt $F^{\prime}_{m,l}$ for $F_{m,l}$. Clearly, $F^{\prime}_{m,1}$ is the graph $S_{\frac{m+4}{2}, 2}^{-}$. Fang and You obtained the following result.

\begin{lemma}[Fang and You \cite{Fang2023}]
\label{lem:F_m,l}
If $l \geq3$ is odd and $m>l+1$, then $\rho(F^{\prime}_{m,l})<\rho(F^{\prime}_{m,1})$.
\end{lemma}

Throughout the rest of this paper, we consider graphs of size $m\geq 38$ and without isolated vertices. In addition, for any $n \geq 1$, denote $[n]=\{1,2,\dots,n\}$. Let $G$ be such a graph with maximum spectral radius among all graphs that do not contain $H(4,3)$ as a subgraph. Firstly, we claim that $G$ is connected. If not, then let $H$ be a component of $G$ which satisfies $\rho\left(G\right)=\rho\left(H\right)$. Add $m-|E(H)|$ leaves to a vertex of $H$. Then we will obtain a new graph $H'\ncong G$ with $\rho\left(H'\right)>\rho\left(H\right)=\rho\left(G\right)$. However, $H'$ does not contain $H(4,3)$ as a subgraph. This is contrary to the assumption of $G$.

Let $\rho =\rho\left(G\right)$ and \textit{\textbf{x}} be the Perron vector of $A(G)$ with coordinates $x_u$ corresponding to the vertices $u \in V\left(G\right)$. Let $u^*$ be a vertex of $V(G)$ such that $x_{u^*}=\max \left\{x_u:\right. \left.u \in V\left(G\right)\right\}$, and $W=V\left(G\right) \backslash N\left[u^*\right]$. We write $N=N\left(u^*\right), A_{+}=\left\{v \in N: d_N(v) \geq\right.$ $1\}$ and $A_0=N \backslash A_{+}$ for short. Then we have
\begin{equation}\label{1}
m=|N|+e\left(A_{+}\right)+e(N, W)+e(W).
\end{equation}
Since $A\left(G\right) \textit{\textbf{x}}=\rho \textit{\textbf{x}}$, we can get
\begin{equation}\label{rhox-max}
\rho x_{u^*}=\left(A\left(G\right) \textit{\textbf{x}}\right)_{u^*}=\sum_{v \in N} x_v=\sum_{v \in A_0} x_v+\sum_{v \in A_{+}} x_v,
\end{equation}
and
\begin{equation}\label{rho2x-max}
\rho^2 x_{u^*}=|N| x_{u^*}+\sum_{v \in A_{+}} d_N(v) x_v+\sum_{w \in W} d_N(w)x_w.
\end{equation}
By the two equations above, it is deduced that
\begin{align*}
\left(\rho^2-\rho \right) x_{u^*}&=|N| x_{u^*}+\sum_{v \in A_{+}}\left(d_N(v)-1\right) x_v+\sum_{w \in W} d_N(w) x_w-\sum_{v \in A_0} x_v.
\end{align*}
Since $\displaystyle \frac{x_v}{x_{u^*}} \leq 1$ for any vertex $v \in V(G)$, and using Handshake Lemma to $G[A_{+}]$, we have
\begin{equation}\label{2}
\rho^2-\rho\leq|N|+2e\left(A_{+}\right)-\left|A_{+}\right|+e(N, W)-\sum_{v \in A_0} \frac{x_v}{x_{u^*}}.
\end{equation}

These several formulas will be frequently used in our proof, and the following lemma shows that each vertex in $W$ has degree at least two.

\begin{lemma}\label{lem2.3}
 Let $G$ be an extremal graph with maximum spectral radius among all $H(4,3)$-free graphs of size $m$. Then $d(w) \geq 2$ for every $w \in W$. 
\end{lemma}

\begin{proof} If there exists a vertex $w\in W$ such that $d(w)=1$ and $wv\in E(G)$, then we shall consider the graph $G'=G-wv+wu^*$. According to $x_{u^*}=\max \left\{x_u: u \in V\left(G\right\}\right.$ and Lemma \ref{lem:Nikiforov}, we have $\rho(G')>\rho(G)$. It is easy to verify that $G'$ is an $H(4,3)$-free graph with size $m$ and without isolated vertices. This leads to a contradiction.
\end{proof}

\section{Proof of Theorem \ref{thm:even-case}}
\label{sec:proof}

Recall that $G$ is an extremal graph with maximum spectral radius among all $H(4,3)$-free graphs of even size $m \geq 38$, and it is connected. Since $S_{\frac{m+4}{2}, 2}^{-}$ does not contain $H(4,3)$ as a subgraph and by Lemma \ref{lem:extremal-graphs-spectral-radius}, it follows that $\rho=\rho(G) >\frac{1+\sqrt{4 m-5}}{2}$. Then we have
$$\rho^2 -\rho >m-\frac{3}{2}.$$
Combining Eqs. (\ref{1}) and (\ref{2}), we obtain
\begin{equation}\label{3}
0\leq e(W)<e\left(A_{+}\right)-\left|A_{+}\right|+\frac{3}{2}-\sum_{v \in A_0} \frac{x_v}{x_{u^*}}.
\end{equation}
Hence, we get
 \begin{equation}\label{6}
0 \leq\sum_{v \in A_0} x_v<\Big[e\left(A_{+}\right)-\left|A_{+}\right|+\frac{3}{2}-e(W)\Big]x_{u^*}.
\end{equation}
By the definition of $x_{u^*}$, and Eqs. (\ref{rhox-max}) and (\ref{6}), we have
$$
\begin{aligned}
\rho^2  x_{u^*} & =\rho \left(\sum_{v \in A_+} x_{v}+\sum_{v \in A_0} x_v\right) =\sum_{v \in A_+} \left(\sum_{u \in N(v)}x_u\right)+\rho \left(\sum_{v \in A_0} x_v\right)\\
&<\left[2e\left(A_{+}\right)+\left|A_{+}\right|+e(A_{+},W)\right]x_{u^*}+\rho \Big[e\left(A_{+}\right)-\left|A_{+}\right|+\frac{3}{2}-e(W)\Big]x_{u^*}.
\end{aligned}
$$
Thus we get
 \begin{equation}\label{4}
\rho^2 -\rho \Big[e\left(A_{+}\right)-\left|A_{+}\right|+\frac{3}{2}-e(W)\Big]<2e\left(A_{+}\right)+\left|A_{+}\right|+e(A_{+},W).
\end{equation} 
We have the following claims.
\begin{claim}
\label{claim:nonempty}
$A_{+} \neq \emptyset$
\end{claim}

\begin{proof}
Suppose that $A_{+} = \emptyset$. Then $G[N]$ consists of isolated vertices only, that is, $e(N)=0$. Since $x_v \leq x_{u^*}$ for each $v \in V(G)$, and by Eqs. (\ref{1}) and (\ref{rho2x-max}), we have
$$
\rho^2  x_{u^*} \leq \Big[|N|+\sum_{v \in A_{+}} d_N(v)+\sum_{w \in W} d_N(w)\Big]x_{u^*} \leq \Big[|N|+e(N, W)+e(W)\Big] x_{u^*}=mx_{u^*}.
$$
Since $m \geq 38$, this further gives $\displaystyle \rho  \leq \sqrt{m}<\frac{1+\sqrt{4 m-5}}{2}$. This leads to a contradiction.
\end{proof}

\begin{claim}
    \label{claim:connected}
    $G[A_{+}]$ is connected.
\end{claim}

\begin{proof}
Since $G$ is $H(4,3)$-free, it is easy to see that $G[A_{+}]$ is a $(P_{2}\cup P_{3})$-free graph. If $G[A_{+}]$ is disconnected, then by Claim \ref{claim:nonempty}, $G[A_{+}]$ is isomorphic to $kP_{2}$ for $k\geq 2$. Due to Eq. (\ref{3}) and $e(A_+)-|A_+| = -k \leq -2$, we have
\begin{equation*}
e(W)<e\left(A_{+}\right)-\left|A_{+}\right|+\frac{3}{2}-\sum_{v \in A_0} \frac{x_v}{x_{u^*}}\leq -2+\frac{3}{2}-\sum_{v \in A_0} \frac{x_v}{x_{u^*}}<0,
\end{equation*}
and this contradicts the fact $e(W) \geq 0$.
\end{proof}

Let $H_1,H_2,\dots,H_7$ be the graphs shown in Figure \ref{fig1}. The following result is derived from Claim \ref{claim:connected} and the observation that $G[A_{+}]$ is a $(P_{2}\cup P_{3})$-free graph.
\begin{claim}
    \label{claim:Hi}
    $G[A_{+}] \cong H_i$ for some $i \in [7]$. 
\end{claim}

\begin{figure}[H]
  \centering
    \includegraphics[width=9.5cm]{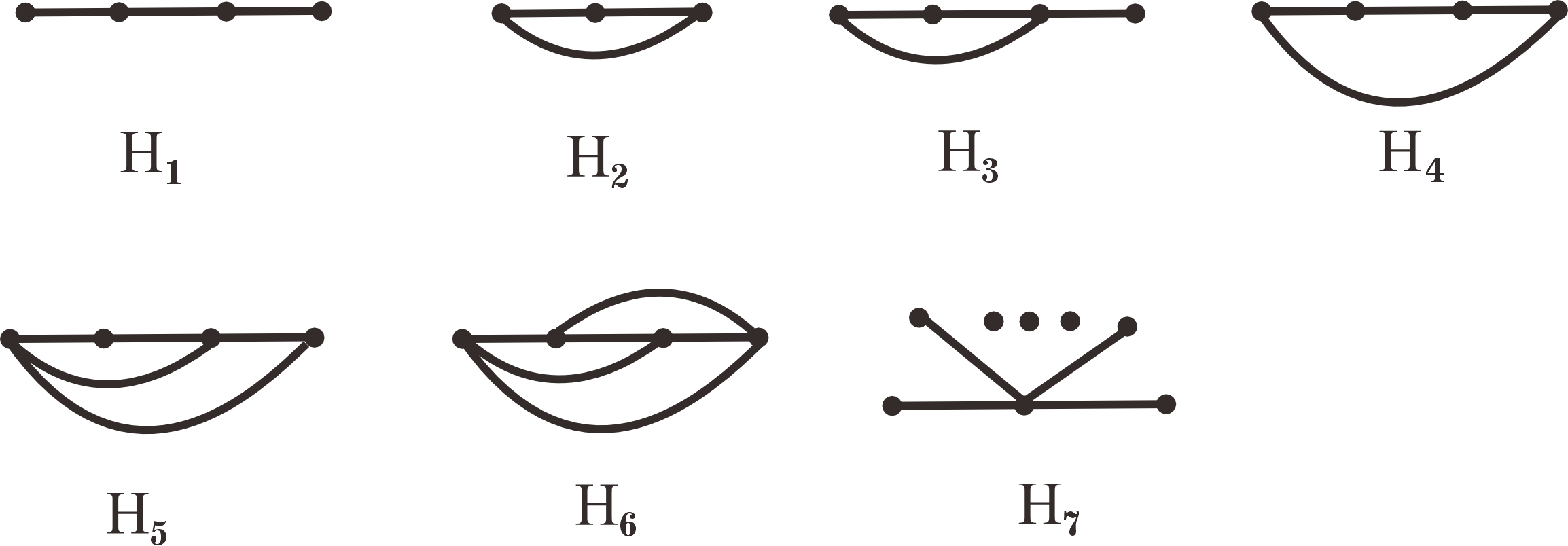}
  \caption{\small The graphs $H_1,H_2,\dots,H_7$.}
  \label{fig1}
\end{figure}

\begin{claim}
\label{claim:A+:1-6}
    $G[A_{+}]\ncong H_{i}$ for each $i \in [6]$.
\end{claim}

\begin{proof}
For a contradiction, suppose that $G[A_{+}]\cong H_{i}$ for some $i \in [6]$. Due to Eq. (\ref{3}), we have
\begin{equation}\label{e(W)}
0\leq e(W)\leq\left\{\begin{array}{l}
0, \ \ \ \text{if}~ G[A_{+}]\cong H_{1}, \\
1,\ \ \ \text{if}~ G[A_{+}]\in \{H_{2},H_{3},H_{4}\},\\
2,\ \ \ \text{if}~ G[A_{+}]\cong H_{5},\\
3,\ \ \ \text{if}~ G[A_{+}]\cong H_{6}. \end{array}\right.
\end{equation} 
We distinguish our proof into the following two cases. 

\vspace{3mm}
\noindent \textbf{Case 1.} $e(A_+,W) \leq 3$.
\vspace{3mm}

By Eq. (\ref{e(W)}) and $e(A_+,W) \leq 3$, we have $\displaystyle e\left(A_{+}\right)-\left|A_{+}\right|+\frac{3}{2}-e(W) \leq \frac{7}{2}$ and $2e\left(A_{+}\right)+\left|A_{+}\right|+e(A_{+},W) \leq 19$. Substituting them into Eq. (\ref{4}), we have
\begin{equation*}
\rho^2 -\frac{7\rho}{2}<19. 
\end{equation*}
However, since $m \geq 38$ and $\rho >\frac{1+\sqrt{4 m-5}}{2}$, we have
\begin{equation*}
19<\left(\frac{1+\sqrt{4 m-5}}{2}\right)^2 -\frac{7(1+\sqrt{4 m-5})}{4}<\rho^2 -\frac{7\rho}{2}<19,
\end{equation*}
a contradiction.


\vspace{3mm}
\noindent \textbf{Case 2.} $e(A_+,W) \geq 4$.
\vspace{3mm}

By $e(A_+,W) \geq 4$, we have $W \neq \emptyset$. Let $w$ be an arbitrary vertex of $W$. By the structure of $G[A_+]$, and since $G$ is $H(4,3)$-free, it is easy to see that,
\begin{enumerate}[label=(\roman*), ref=(\roman*)]
    \item \label{item:first} $d_{A_0}(w) \leq 1$,

   \item \label{item:second} if $d_{A_0}(w)\geq 1$, then $d_{A_+}(w)=0$,

    \item \label{item:third} and if $d_W(w)=0$, then we have $d_{A_0}(w) =0$ and $d(w)=d_{A_+}(w) \geq 2$ by Lemma \ref{lem2.3} and items \ref{item:first}, \ref{item:second}.
\end{enumerate}

\noindent Moreover, let $w_1,w_2$ and $v_1,v_2$ be arbitrary two vertices in $W$ and $A_+$, respectively. It is easy to check

\begin{enumerate}[label=(\roman*), ref=(\roman*), resume]
    \item \label{item:fourth} if $w_1w_2 \in E(G)$, then $d_{\{w_1,w_2\}}(v) \leq 1$ for each $v \in A_+$,
    
    \item \label{item:fifth} and if $w_1w_2,v_1v_2 \in E(G)$, then there exists $i \in [2]$ such that $d_{\{v_1,v_2\}}(w_i)=0$ or $d_{\{w_1,w_2\}}(v_i)=0$.
\end{enumerate}

\vspace{3mm}
\noindent \textbf{Subcase 2.1.} For some $w^* \in W$, $d_{A_+}(w^*) \geq 2$.
\vspace{3mm}

If $G[A_+] \cong H_1 \cong P_4$, then we assume that $H_1:=v_1v_2v_3v_4$. One can easily verify that $d_{A_+}(w^*) = 2$ and $N_{A_+}(w^*)=\{v_2,v_3\}$. Since $e(W)=0$ and $e(A_+,W) \geq 4$, there exists another vertex $w' \in W$ with $N_{A_+}(w')=\{v_2,v_3\}$. Now $\{u^*,v_1,v_2,w^*,w',v_3\}$ forms a graph which contains $H(4,3)$ as a subgraph, a contradiction.

If $G[A_+] \cong H_2 \cong C_3$, then $2 \leq d_{A_+}(w^*) \leq 3$. Suppose that there exists another vertex $w' \in W$ with $d_{A_+}(w') \geq 2$. Then $w^*$ and $w'$ at least have a common neighbor in $A_+$, and $\{w^*,w',u^*\} \cup A_+$ forms a graph that contains $H(4,3)$ as a subgraph, a contradiction. So, for each $w \in W\setminus \{w^*\}$, we may assume that $d_{A_+}(w) \leq 1$. Since $e(A_+,W) \geq 4$, there must exist a vertex $w_1 \in W\setminus \{w^*\}$ with $d_{A_+}(w_1) = 1$. By Items \ref{item:second} and \ref{item:fifth}, we know $d_{A_0}(w_1)=0$ and $w^*w_1 \notin E(G)$. Combining Lemma \ref{lem2.3}, Item \ref{item:third} and Eq. (\ref{e(W)}), it is derived that $W=\{w^*,w_1,w_2\}$, where $w_2$ is the third vertex of $W$ and $w'w'' \in E(G)$, i.e., $e(W)=1$. Hence, $e(A_+,W) = 4$ and $d_{A_+}(w^*)=3$. Furthermore, Item \ref{item:first} and Lemma \ref{lem2.3} imply $d_{A_0}(w_2)=1$ and $|A_0| \geq 1$. Now the graph $G$ is determined, see Figure \ref{fig2}, $G_1$.

Then, by Eq. (\ref{4}), and since $m \geq 38$ and $\rho >\frac{1+\sqrt{4 m-5}}{2}$, we get \begin{equation}\label{eq:contradiction-13}
13<\left(\frac{1+\sqrt{4 m-5}}{2}\right)^2 -\frac{1+\sqrt{4 m-5}}{4}<\rho^2 -\frac{\rho}{2}<13,
\end{equation}
a contradiction.

If $G[A_+] \cong H_i$ for $i \in \{3,4,5,6\}$, then $\{w^*,u^*\} \cup A_+$ forms a graph which contains $H(4,3)$ as a subgraph. This leads to a contradiction. 

\begin{figure}[H]
  \centering
    \includegraphics[width=8cm]{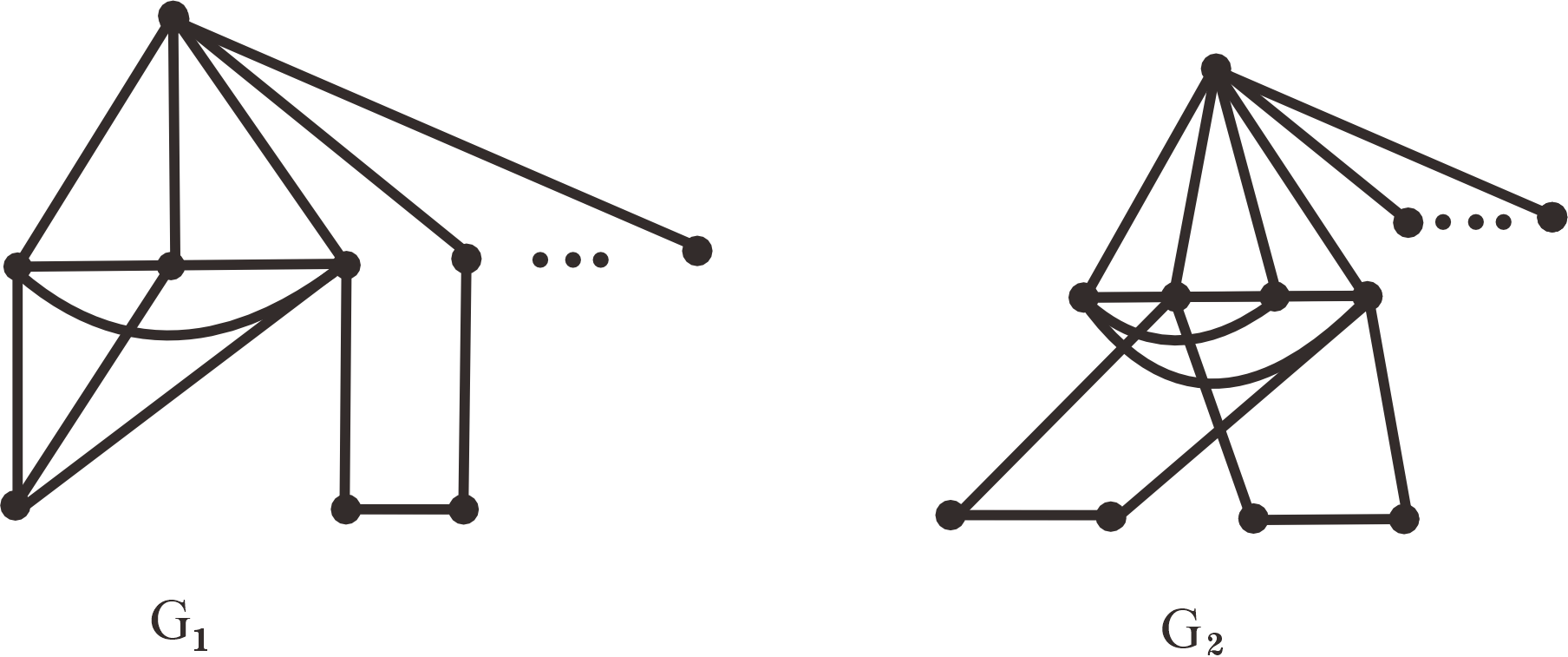}
  \caption{\small The graphs $G_{1}$ and $G_2$.}
  \label{fig2}
\end{figure}

\vspace{3mm}
\noindent \textbf{Subcase 2.2.} For each $w \in W$, $d_{A_+}(w) \leq 1$.
\vspace{3mm}

By Lemma \ref{lem2.3} and the present assumption, we know that $d_W(w) \geq 1$ for each $w \in W$. Together with $e(A_+,W) \geq 4$, we have $|W| \geq 4$ and $e(W) \geq 2$. Thus, $G[A_{+}] \in \{H_5.H_6\}$. 

If $G[A_{+}] \cong H_5 \cong K_4-e$, by Eq. (\ref{e(W)}), $e(W) = 2$, and we determine $G[W] \cong 2K_2$. For each $w \in W$, we have $d_{A_+}(w) = 1$. By Item \ref{item:second}, $d_{A_0}(w) = 0$. According to Items \ref{item:fourth} and \ref{item:fifth}, the two vertices of each $K_2$ of $G[W]$ has no a common neighbor in $A_+$, and they can only join the two vertices in the unique largest independent set of $A_+$ separately. Now the graph $G$ is determined, see Figure \ref{fig2}, $G_2$. Note that $A_0=\emptyset$ is possible.

Then, by Eq. (\ref{4}) and similar to Eq. (\ref{eq:contradiction-13}), we get \begin{equation*}
18<\left(\frac{1+\sqrt{4 m-5}}{2}\right)^2 -\frac{1+\sqrt{4 m-5}}{4}<\rho^2 -\frac{\rho}{2}<18,
\end{equation*}
again a contradiction. 

If $G[A_{+}] \cong H_6 \cong K_4$, then by Eq. (\ref{e(W)}), $e(W) \leq 3$. Further, we can determine $G[W] \in \{2K_2,3K_2,K_{1,3},P_4,K_2 \cup P_3\}$. There exist two vertices $w_1,w_2 \in W$ such that $w_1w_2 \in E(G)$ and $d_{A_+}(w_i) = 1$ for each $i \in [2]$, a contradiction to Item \ref{item:fifth}. 
This proves Claim \ref{claim:A+:1-6}.
\end{proof}



 

\begin{claim}
\label{claim:H7}
    If $G[A_{+}]\cong H_{7}$, then $W =\emptyset$.
\end{claim}

\begin{proof}
For a contradiction, suppose that $W \neq \emptyset$.
Since $G[A_{+}]\cong H_{7} \cong K_{1,t}$ for some $t \geq 1$, we have $e\left(A_{+}\right)-\left|A_{+}\right|=-1$. By Eq. (\ref{3}), we obtain $e(W)=0$, and according to Eq. (\ref{4}),
\begin{equation}\label{5}
\rho^2 -\frac{\rho}{2}<3t+1+e(A_{+},W).
\end{equation}

First, we prove that $A_{0}\neq\emptyset$. Then we will get $|N| > t + 1$.  Assume that $A_{0}=\emptyset$. For the case $t =1$, since $m \geq38$ and $e(W) = 0$, by Lemma \ref{lem2.3}, we can deduce that $x_{v}>x_{u^*}$ for each vertex $v\in N(u^*)$. This is a contradiction to $x_{u^*}=\max \left\{x_u:\right.$ $\left.u \in V\left(G\right)\right\}$. For the case $t = 2$, by Lemma \ref{lem2.3}, we get $|W| \leq1$, and now the graph $G$ can only have at most 8 edges, a contradiction to $m\ge 38$. It remains to consider the case $t\geq3$. To forbid the subgraph $H(4, 3)$ in $G$, each vertex in $W$ can only have at most one neighbor in $N(u^*)$. Combining with Lemma \ref{lem2.3}, we have $W=\emptyset$, again a contradiction to $W \neq \emptyset$. 

Thus we have $|N| > t + 1$. If $t=1$, then let $v_1 v_2$ be the unique edge of $G[N]$. Observe that $d_{A_0}(w) \leq 1$ for any vertex $w \in W$, and thus, $d_{A_{+}}(w) \geq 1$ and $d_N(w)\leq3$. Moreover, there is at most one vertex $w \in W$ with $d_N(w)=3$. Assume that there exists a vertex $w \in W$ with $d_N(w)=3$, and let $N(w)=\left\{u, v_1, v_2\right\}$ where $u \in A_0$. Consider the graph $G'=G-w u+w u^*$. It is easy to see that $G'$ has size $m$ and satisfies the condition of Lemma \ref{lem:Nikiforov}, so that $\rho(G')>\rho(G)$ and we obtain a contradiction. By Lemma \ref{lem2.3}, we may assume that $d_N(w)=2$ for any $w \in W$.  

Furthermore, we get $e(N, W)=2|W|$. It follows that $e(A_{+},W)\leq2|W|$ and 
$$
|W|=\frac{m-|N|-1}{2}.
$$
Using Eq. (\ref{5}), we have
$$
\rho^2 -\frac{\rho}{2}<3+1+(m-|N|-1).
$$
However, due to $|N| \geq 3$ and $m\geq 38$,
$$
3+1+(m-|N|-1)<\frac{(1+\sqrt{4 m-5})\sqrt{4 m-5}}{4},
$$
and we will obtain 
$$
\rho^2 -\frac{\rho}{2}<\frac{(1+\sqrt{4 m-5})\sqrt{4 m-5}}{4},
$$
which contradicts  $\rho >\frac{1+\sqrt{4 m-5}}{2}$.

If $t=2$, then $d_N(w)\leq3$ for any vertex $w \in W$. Assume that there exists a vertex $w \in W$ with $d_N(w)=3$. Now the graph $G$ can only be isomorphic to $G_3$, see Figure \ref{fig33}. By Eq. (\ref{5}), we obtain
$$
\rho^2 -\frac{\rho}{2}<10.
$$
Since $m\geq 38$ and $\rho >\frac{1+\sqrt{4 m-5}}{2}$, we get a contradiction that
$$
10<\left(\frac{1+\sqrt{4 m-5}}{2}\right)^2 -\frac{1+\sqrt{4 m-5}}{4}<\rho^2 -\frac{\rho}{2}<10.
$$

\begin{figure}[H]
  \centering
    \includegraphics[width=3cm]{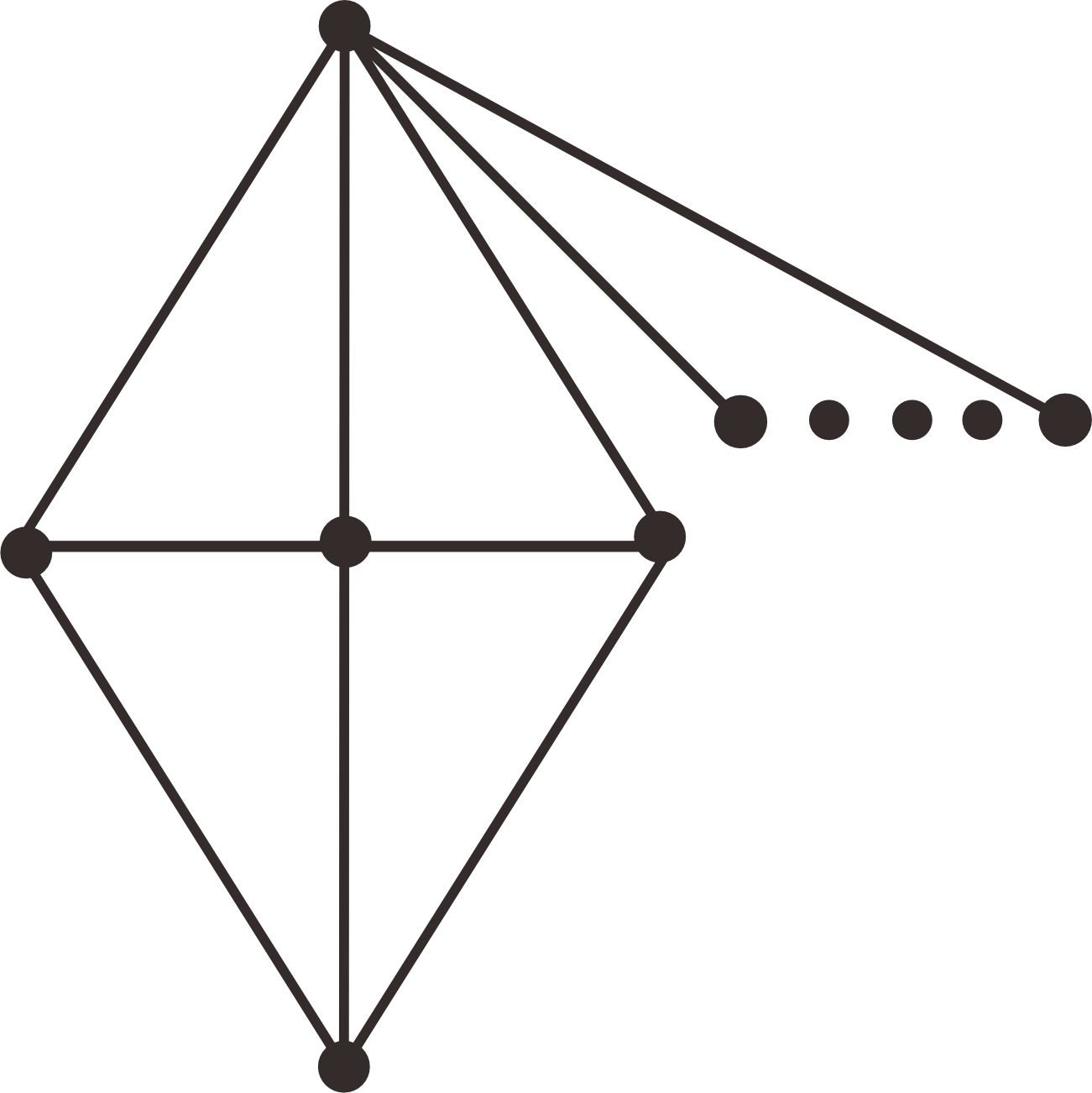}
  \caption{\small The graph $G_{3}$.}
  \label{fig33}
\end{figure}

We have $d_N(w)=2$ for any $w \in W$. By $e(W)=0$ and Lemma \ref{lem2.3}, it deduces that $m-|N|-2=e(N, W)=2|W|$, and we get
$$
|W|=\frac{m-|N|-2}{2}.
$$
It is easy to check that there is at most one vertex $w$ in $W$ with $d_{A_+}(w)=2$, based on the forbidden $H(4,3)$ in $G$. Thus, we have $e(A_{+},W)\leq|W|+1$. According to Eq. (\ref{5}), we get
$$
\rho^2 -\frac{\rho}{2}<6+1+\frac{m-|N|-2}{2}+1<\frac{(1+\sqrt{4 m-5})\sqrt{4 m-5}}{4}
$$
due to $|N| \geq 3$ and $m\geq 38$, which is a contradiction to $\rho >\frac{1+\sqrt{4 m-5}}{2}$.

\begin{figure}[H]
  \centering
    \includegraphics[width=5cm]{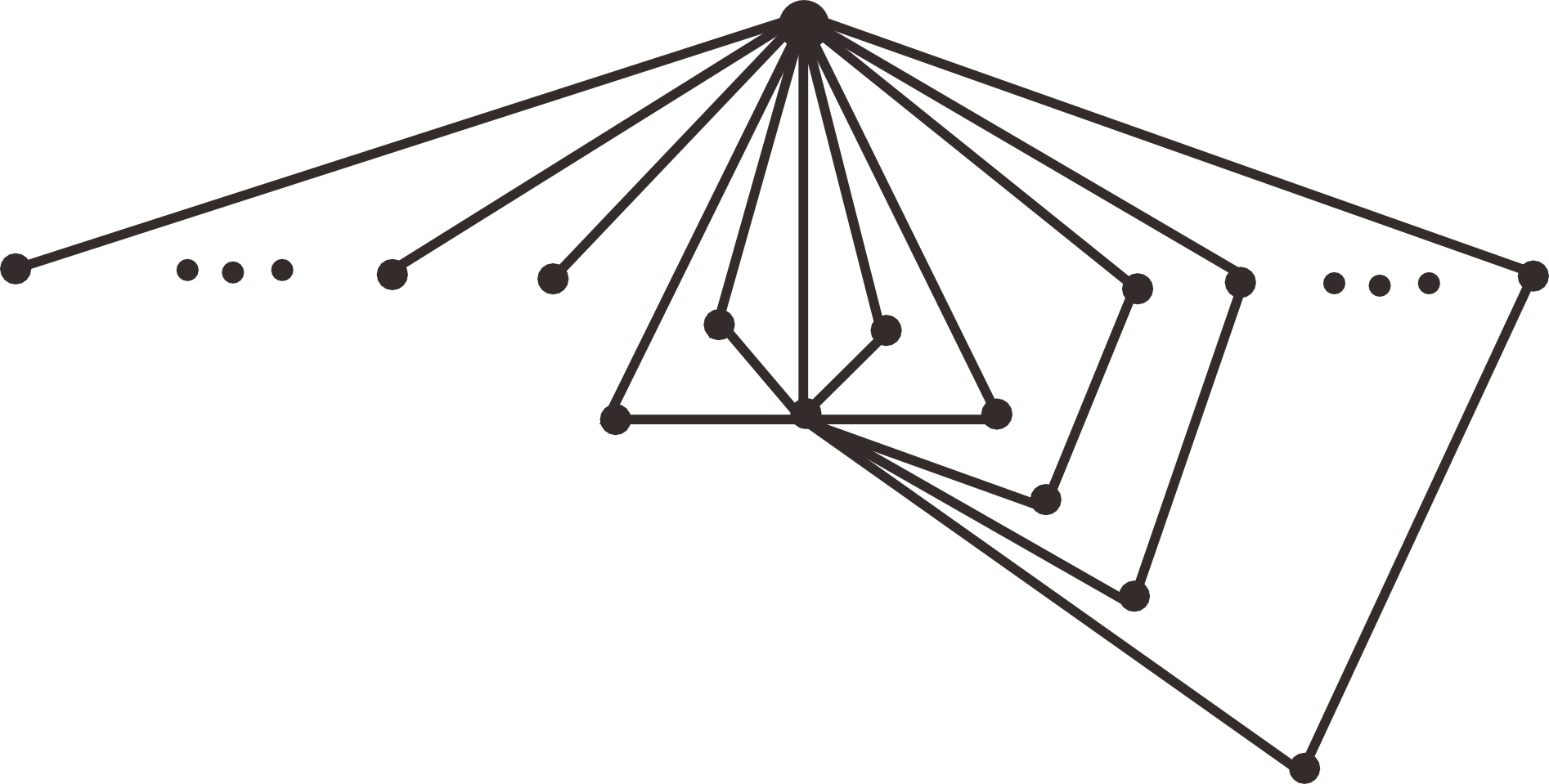}
  \caption{\small The graph $G_{4}$.}
  \label{fig4}
\end{figure}

For the case $t\geq3$, since $W\neq\emptyset$ and $G$ dose not contain $H(4,3)$ as a subgraph, the graph $G$ must isomorphic to $G_4$, see Figure \ref{fig4}. Let $W=\{w_{1}, \ldots, w_{s}\}$, and let $N(w_{i})=\left\{v, u_i\right\}$ for $i \in [s]$, where $v$ is the center of $G[A_+]$ and $u_i \in A_0$. Then the graph $G''=G_{4}-w_{1}u_{1}-\ldots-w_{s}u_{s}+w_{1}u^*+\dots+w_{s}u^*$ has size $m$ and satisfies the condition of Lemma \ref{lem:Nikiforov}. Hence we have $\rho(G'')>\rho(G_{4})=\rho(G)$, and this leads to a contradiction again.
\end{proof}
By Claim \ref{claim:H7}, we conclude that $G[A_{+}]\cong G_{7}$ and $W=\emptyset$. Now the graph $G$ is determined, that is, $G \cong F^{\prime}_{m,l}$ for some integer $l \geq 0$. Since $m$ is even, $l \geq 1$ and $l=m-1-2t$ is odd. By Lemma \ref{lem:F_m,l}, we know that $\rho(F^{\prime}_{m,l})\leq \rho(F^{\prime}_{m,1})$ with equality if and only if $l=1$. Then we obtain the extremal graph $G \cong F^{\prime}_{m,1} \cong S_{\frac{m+4}{2}, 2}^{-}$. This completes the proof of Theorem \ref{thm:even-case}.





\section{Conclusions}

Recall that Li, Lu and Peng \cite{LiLuPeng2023} and Chen et al. \cite{Chen2025} determined the maximum spectral radius of $H(3,3)$-free graphs with given odd and even size, respectively. The spectral extrema result for the graph $F_5$ by Gao and Li \cite{Gao2026} implies the result for $H(4,3)$-free graphs with given odd size. In this paper, we prove that if $G$ is an $H(4,3)$-free graph of even size $m \geq 38$ without isolated vertices, then $\rho(G) \leq \rho^{\prime}(m)$ with equality if and only if $G \cong S_{\frac{m+4}{2}, 2}^{-}$, where $\rho^{\prime}(m)$ is the largest root of $x^4-m x^2-(m-2) x+\frac{m}{2}-1=0$. It is natural to consider the following problem.

\begin{problem}
    Let $k,l \geq 3$ be two integers. Determine the exact maximum spectral radius of an $F_5$-free graph of given even size $m$.
\end{problem}

\end{document}